\documentclass{article}

\usepackage{amssymb, amsmath, amsthm, tikz, verbatim, graphicx,lmodern,indentfirst,enumerate,color}

\usepackage[titletoc,page]{appendix}
\usepackage[margin=1.4in]{geometry}

\newtheorem{thm}{Theorem}[section]
\newtheorem{lem}[thm]{Lemma}
\newtheorem{prop}[thm]{Proposition}
\newtheorem{cor}[thm]{Corollary}

\newtheorem{oprob}[thm]{Open Problem}

\theoremstyle{definition}

\theoremstyle{definition}
\newtheorem{defn}[thm]{Definition}

\theoremstyle{remark}
\newtheorem{rem}[thm]{Remark}

\numberwithin{equation}{section}

\makeatletter
\newcommand{\rmnum}[1]{\romannumeral #1}
\newcommand{\Rmnum}[1]{\expandafter\@slowromancap\romannumeral #1@}
\makeatother
\begin{document}

\title{Complete reducibility of subgroups of reductive algebraic groups over nonperfect fields \Rmnum{2}}
\author{Tomohiro Uchiyama\\
National Center for Theoretical Sciences, Mathematics Division\\
No.~1, Sec.~4, Roosevelt Rd., National Taiwan University, Taipei, Taiwan\\
\texttt{email:t.uchiyama2170@gmail.com}}
\date{}
\maketitle 

\begin{abstract}
Let $k$ be a separably closed field. Let $G$ be a reductive algebraic $k$-group. We study Serre's notion of complete reducibility of subgroups of $G$ over $k$. In particular, using the recently proved center conjecture of Tits, we show that the centralizer of a $k$-subgroup $H$ of $G$ is $G$-completely reducible over $k$ if it is reductive and $H$ is $G$-completely reducible over $k$. We show that a regular reductive $k$-subgroup of $G$ is $G$-completely reducible over $k$. We present examples where the number of overgroups of irreducible subgroups and the number of $G(k)$-conjugacy classes of $k$-anisotropic unipotent elements are infinite.
\end{abstract}

\section{Introduction} 
Let $k$ be an arbitrary field. Let $\overline k$ be an algebraic closure of $k$. Let $H$ be a (possibly non-connected) affine algebraic $k$-group, that is a (possibly non-connected) $\overline k$-defined affine algebraic group with a $k$-structure in the sense of Borel~\cite[AG.12.1]{Borel-AG-book}. We write $R_{u,k}(H)$ for the unique maximal smooth connected unipotent normal $k$-subgroup of $H$. An affine algebraic $k$-group $H$ is \emph{pseudo-reductive} if $R_{u,k}(H)=1$~\cite[Def.~1.1.1]{Conrad-pred-book}, and \emph{reductive} if the unipotent radical $R_u(H)=1$. 
Throughout, we write $G$ for a (possibly non-connected) reductive algebraic $k$-group.  Following Serre~\cite[Sec.~3]{Serre-building}, define
\begin{defn}\label{G-cr}
A (possibly non-$k$-defined) closed subgroup $H<G$ is \emph{$G$-completely reducible over $k$} ($G$-cr over $k$ for short) if whenever $H$ is contained in a $k$-defined $R$-parabolic subgroup $P$ of $G$, it is contained in some $k$-defined $R$-Levi subgroup of $P$. In particular, if $H$ is not contained in any $k$-defined proper $R$-parabolic subgroup, $H$ is \emph{$G$-irreducible} over $k$ ($G$-ir over $k$ for short).
\end{defn}
For the definition of $R$-parabolic subgroups and $R$-Levi subgroups, see Definition~\ref{non-connected-G-cr}. 
If $G$ is connected, $R$-parabolic subgroups and $R$-Levi subgroups are parabolic subgroups and Levi subgroups in the usual sense. Definition~\ref{G-cr} extends usual Serre's definition in the following sense: 1.~$H<G$ is not necessarily $k$-defined, 2.~$G$ is not necessarily connected. Definition~\ref{G-cr} was used in~\cite{Bate-cocharacter-Arx} and~\cite{Uchiyama-Nonperfect-pre}. By a subgroup of $G$, we always mean a closed subgroup of $G$. 

The notion of complete reducibility generalizes that of complete reducibility in representation theory, and it has been much studied. However most studies assume $k=\overline k$ and $G$ is connected; see~\cite{Bate-geometric-Inventione},\cite{Liebeck-Seitz-memoir},\cite{Stewart-nonGcr} for example. We say that $H<G$ is $G$-cr when it is $G$-cr over $\overline k$. Not much is known about complete reducibility over an arbitrary $k$ except a few general results and important examples in~\cite{Bate-cocharacter-Arx},~\cite{Bate-separable-Paris},~\cite[Sec.~7]{Bate-separability-TransAMS},~\cite{Bate-uniform-TransAMS},~\cite{Uchiyama-Nonperfect-pre},~\cite[Thm.~1.8]{Uchiyama-Classification-pre},~\cite[Sec.~4]{Uchiyama-Separability-JAlgebra}. 

Let $k_s$ be a separable closure of $k$. Recall that if $k$ is perfect, we have $k_s=\overline k$. The following result~\cite[Thm.~1.1]{Bate-separable-Paris} shows that if $k$ is perfect and $G$ is connected, most results in this paper just reduce to the algebraically closed case. 
\begin{prop}\label{separable}
Let $k$ be a field. Let $G$ be connected. Then a $k$-subgroup $H$ of $G$ is $G$-cr over $k$ if and only if $H$ is $G$-cr over $k_s$.
\end{prop}


We write $G(k)$ for the set of $k$-points of $G$. For $H<G$, we write $H(k):=G(k)\cap H$. By $\overline H$, we mean the Zariski closure of $H$. We write $C_G(H)$ for the set-theoretic centralizer of $H$ in $G$. Centralizers of subgroups of $G$ are important to understand the subgroup structure of $G$~\cite{Bala-classes-Cam},~\cite{Bala-classes-Cam2}~\cite{Liebeck-Seitz-unipotent},~\cite{SpringerSteinberg-book}. 
Recall the following~\cite[Prop.~3.12, Cor.~3.17]{Bate-geometric-Inventione}: 
\begin{prop}\label{algclocentralizer}
Let $k=\overline{k}$. Suppose that a subgroup $H$ of $G$ is $G$-cr. Then $C_G(H)$ is reductive, and moreover it is $G$-cr. 
\end{prop}
Note that any $G$-cr subgroup of $G$ is reductive~\cite[Prop.~4.1]{Serre-building}. It is natural to ask (cf.~\cite[Open Problem~1.13]{Uchiyama-Nonperfect-pre}):
\begin{oprob}\label{centralizerquestion}
Let $k$ be a field. Suppose that a $k$-subgroup $H$ of $G$ is $G$-cr over $k$. Is $C_G(H)$ $G$-cr over $k$? Is $\overline{C_G(H)(k_s)}$ $G$-cr over $k$? 
\end{oprob} 
Even if $H$ is $k$-defined, $C_G(H)$ is not necessarily $k$-defined; see~\cite[Theorem~1.2]{Uchiyama-Nonperfect-pre} for examples of non-$k$-defined $C_G(H)$. See~Lemma~\ref{centralizerkdefined} and~\cite[Prop.~7.4]{Bate-cocharacter-Arx} for some $k$-definability criteria for $C_G(H)$. Let $\Gamma:=\textup{Gal}(k_s/k)=\textup{Gal}(\overline{k}/k)$. Note that $\overline{C_G(H)(k_s)}$ is the unique maximal $k$-defined subgroup of $C_G(H)$; it is $k$-defined by~\cite[Prop.~14.2]{Borel-AG-book} since it is $k_s$-defined and $\Gamma$-stable. Also note that the slightly different notation $\overline{C_{G(k_s)}(H)}$ was used in~\cite[Lem.~C.4.1]{Conrad-pred-book} for $\overline{C_G(H)(k_s)}$ (they are the same, we follow the notation in~\cite{Bate-cocharacter-Arx}). Our principal result is the following.

\begin{thm}\label{main}
Let $k=k_s$. Let $G$ be connected. Suppose that a $k$-subgroup $H$ of $G$ is $G$-cr over $k$.  
\begin{enumerate}
\item{If $\overline{C_G(H)(k)}$ is pseudo-reductive, then it is $G$-cr over $k$,}
\item{If $C_G(H)$ is reductive, then it is $G$-cr over $k$.}
\end{enumerate}
\end{thm} 

Recall the following~\cite[Prop.~1.14]{Uchiyama-Nonperfect-pre}:
\begin{prop}\label{TCC-centralizer}
Let $k=k_s$. Let $G$ be connected. Let $H$ be a $k$-subgroup of $G$. If $H$ is $G$-ir over $k$, then $C_G(H)$ is $G$-cr over $k$. 
\end{prop} 
Theorem~\ref{main} and Proposition~\ref{TCC-centralizer} give a partial affirmative answer to Open Problem~\ref{centralizerquestion}. However, in~\cite[Rem.~3.9]{Uchiyama-Nonperfect-pre}, it was shown that there exists a $k$-subgroup $H$ of $G$ with the following properties: 1.~$H$ is $G$-cr over $k$, 2.~$C_G(H)$ is $k$-defined, 3.~$C_G(H)$ is not pseudo-reductive. This result suggests a negative answer to Open problem~\ref{centralizerquestion} since reductivity of $C_G(H)$ was crucial to show that $C_G(H)$ is $G$-cr in the proof of Proposition~\ref{algclocentralizer}.

Although we cannot solve Open problem~\ref{centralizerquestion} completely, using various group theoretic arguments we extend a number of results in~\cite{Bate-geometric-Inventione} on the structure of $G$-cr subgroups to a separably closed $k$, see~Theorems~\ref{stronglyreductive} and~\ref{regular} for example. 

Note that in Propositions~\ref{algclocentralizer},~\ref{TCC-centralizer} and Theorems~\ref{main},~\ref{regular} we assume $G$ to be connected. This is because the proofs of these results depend on the following (Theorem~\ref{TCCgroup}) that is a consequence of the recently proved \emph{center conjecture of Tits}~\cite{Leeb-Ramos-TCC-GFA},~\cite{Muhlherr-Tits-TCC-JAlgebra},~\cite{Ramos-centerconj-Geo},~\cite[Sec.~2.4]{Serre-building},~\cite[Lem.~1.2]{Tits-colloq},~\cite[Sec.~3.1]{Uchiyama-Nonperfect-pre}.  


\begin{thm}\label{TCCgroup}
Let $k$ be a field. Let $G$ be connected. Let $\Delta(G)$ be the spherical building of $G$. Let $H$ be a (possibly non-$k$-defined) subgroup of $G$ that is not $G$-cr over $k$. Let $\Delta(G)^H$ be the fixed point subcomplex of $\Delta(G)$. Then there exists a simplex in $\Delta(G)^H$ that is fixed by all building automorphisms of $\Delta(G)$ stabilizing $\Delta(G)^H$. 
\end{thm}

Recall that each simplex of $\Delta(G)$ is identified with a proper $k$-parabolic subgroup of $G$~\cite[Thm.~5.2]{Tits-book}, and $\Delta(G)^H$ is identified with the set of $k$-parabolic subgroups containing $H$~\cite[Sec.~2]{Serre-building}. Now let $G$ be non-connected reductive. Note that $\Delta(G)=\Delta(G^{\circ})$ by definition. Let $\Lambda(G)$ be the set of $k$-defined $R$-parabolic subgroups of $G$. Let $\Lambda(G)^{H}$ be the set of $k$-defined $R$-parabolic subgroups of $G$ containing $H$. 
\begin{thm}\label{TCCnotsubset}
Let $k$ be a field. Let $\tilde G=SL_3$. Let $G=\tilde G\rtimes \langle \sigma \rangle$ where $\sigma$ is the non-trivial graph automorphism of $\tilde G$. Then there exists a $k$-subgroup $H$ of $G$ such that $H$ is not $G$-cr over $k$ and $\Lambda(G)^H$ is not a subset of $\Delta(G)$. Moreover, $\Lambda(G)$ ordered by reversed inclusion does not form a simplicial complex in the sense of~\cite[Thm.~5.2]{Tits-book}. 
\end{thm}

Theorem~\ref{TCCnotsubset} shows that we cannot use~Theorem~\ref{TCCgroup} to extend Propositions~\ref{algclocentralizer},~\ref{TCC-centralizer} and Theorems~\ref{main},~\ref{regular} to non-connected $G$. Before finishing this section, we consider a problem with a slightly different flavor. In~\cite[Thm.~1]{Liebeck-Testerman-irreducible-QJM}, Liebeck and Testerman showed that:
\begin{prop}\label{overgroups}
Let $k=\overline k$. Let $G$ be a semisimple. Suppose that $H$ is a connected subgroup of $G$ and $H$ is $G$-ir. Then $H$ has only finitely many overgroups.
\end{prop}
We show that:
\begin{thm}\label{infinitelymanyconjugacy}
Let $k$ be a nonperfect field of characteristic $2$. Let $G:=PGL_4$. Then there exists a connected $k$-subgroup $H$ of $G$ such that $H$ is $G$-ir over $k$ and $H$ has infinitely many (non-$k$-defined) overgroups.
\end{thm}
 
Here is the structure of the paper. In Section 2, we set out the notation. In Sections 3 and 4, we prove Theorems~\ref{stronglyreductive} and~\ref{regular}, respectively. Then in Section 5, we discuss some relationships between complete reducibility and linear reductivity. In Section 6, we prove Theorem~\ref{main}. In Sections 7 and 8, we prove Theorems~\ref{TCCnotsubset} and~\ref{infinitelymanyconjugacy}, respectively. Finally, in Section 9, we present an example where the number of $G(k)$-conjugacy classes of $k$-anisotropic unipotent elements of $G(k)$ is infinite. This paper complements the author's previous work on rationality problems for complete reducibility~\cite{Uchiyama-Nonperfect-pre}.

\section{Preliminaries}
Throughout, we denote by $k$ an arbitrary field. Our basic references for algebraic groups are~\cite{Borel-AG-book},~\cite{Borel-Tits-Groupes-reductifs},~\cite{Conrad-pred-book}, and~\cite{Springer-book}. We write $G$ for a (possibly non-connected) reductive $k$-group. We write $X_k(G)$ and $Y_k(G)$ for the set of $k$-characters and $k$-cocharacters of $G$ respectively. For an algebraic group $H$, we write $H^{\circ}$ for the identity component of $H$. 

Fix a maximal $k$-split torus $T$ of $G$. We write $\Psi_k(G,T)$ for the set of $k$-roots of $G$ with respect to $T$~\cite[21.1]{Borel-AG-book}. We sometimes write $\Psi_k(G)$ for $\Psi_k(G,T)$. Let $\zeta\in\Psi_k(G)$. We write $U_\zeta$ for the corresponding root subgroup of $G$. Let $\zeta, \xi \in \Psi_k(G)$. Let $\xi^{\vee}$ be the coroot corresponding to $\xi$. Then $\zeta\circ\xi^{\vee}:\overline{ k}^{*}\rightarrow \overline{k}^{*}$ is a $k$-homomorphism such that $(\zeta\circ\xi^{\vee})(a) = a^n$ for some $n\in\mathbb{Z}$.
Let $s_\xi$ denote the reflection corresponding to $\xi$ in the Weyl group $W_k$ relative to $T$. Each $s_\xi$ acts on the set of roots $\Psi_k(G)$ by the following formula~\cite[Lem.~7.1.8]{Springer-book}:
$
s_\xi\cdot\zeta = \zeta - \langle \zeta, \xi^{\vee} \rangle \xi. 
$
\noindent By \cite[Prop.~6.4.2, Lem.~7.2.1]{Carter-simple-book} we can choose $k$-homomorphisms $\epsilon_\zeta : \overline{k} \rightarrow U_\zeta$  so that 
$
n_\xi \epsilon_\zeta(a) n_\xi^{-1}= \epsilon_{s_\xi\cdot\zeta}(\pm a)
            \text{ where } n_\xi = \epsilon_\xi(1)\epsilon_{-\xi}(-1)\epsilon_{\xi}(1).  \label{n-action on group}
$

We are going to recall below the notions of $R$-parabolic subgroups and $R$-Levi subgroups from~\cite[Sec.~2.1--2.3]{Richardson-conjugacy-Duke}. These notions are essential to define $G$-complete reducibility for subgroups of non-connected reductive groups; see~\cite{Bate-nonconnected-PAMS} and~\cite[Sec.~6]{Bate-geometric-Inventione}. First we need a definition of a limit. 

\begin{defn}
Let $X$ be a $k$-affine variety. Let $\phi : \overline{k}^*\rightarrow X$ be a $k$-morphism of $k$-affine varieties. We say that $\displaystyle\lim_{a\rightarrow 0}\phi(a)$ exists if there exists a $k$-morphism $\hat\phi:\overline{k}\rightarrow X$ (necessarily unique) whose restriction to $\overline{k}^{*}$ is $\phi$. If this limit exists, we set $\displaystyle\lim_{a\rightarrow 0}\phi(a) = \hat\phi(0)$.
\end{defn}

\begin{defn}\label{non-connected-G-cr}
Let $\lambda\in Y_k(G)$. Define
$
P_\lambda := \{ g\in G \mid \displaystyle\lim_{a\rightarrow 0}\lambda(a)g\lambda(a)^{-1} \text{ exists}\}, $\\
$L_\lambda := \{ g\in G \mid \displaystyle\lim_{a\rightarrow 0}\lambda(a)g\lambda(a)^{-1} = g\}, \,
R_u(P_\lambda) := \{ g\in G \mid  \displaystyle\lim_{a\rightarrow0}\lambda(a)g\lambda(a)^{-1} = 1\}
$. We call $P_\lambda$ an $R$-parabolic subgroup of $G$, $L_\lambda$ an $R$-Levi subgroup of $P_\lambda$. Note that $R_u(P_\lambda)$ is the unipotent radical of $P_\lambda$.
\end{defn}
If $\lambda$ is $k$-defined, $P_\lambda$, $L_\lambda$, and $R_u(P_\lambda)$ are $k$-defined~\cite[Lem.~2.5]{Bate-uniform-TransAMS}. It is well known that $L_\lambda = C_G(\lambda(\overline k^*))$. Note that if $k=k_s$, for a $k$-defined $R$-parabolic subgroup $P$ of $G$ and a $k$-defined $R$-Levi subgroup $L$ of $P$ there exists $\lambda\in Y_k(G)$ such that $P=P_\lambda$ and $L=L_\lambda$~\cite[Lem.~2.5, Cor.~2.6]{Bate-uniform-TransAMS}. 

Let $M$ be a reductive $k$-subgroup of $G$. Then there is a natural inclusion $Y_k(M)\subseteq Y_k(G)$ of $k$-cocharacters. Let $\lambda\in Y_k(M)$. We write $P_\lambda(G)$ or just $P_\lambda$ for the $k$-defined $R$-parabolic subgroup of $G$ corresponding to $\lambda$, and $P_\lambda(M)$ for the $k$-defined $R$-parabolic subgroup of $M$ corresponding to $\lambda$. It is clear that $P_\lambda(M) = P_\lambda(G)\cap M$ and $R_u(P_\lambda(M)) = R_u(P_\lambda(G))\cap M$. 

The next result is a consequence of Theorem~\ref{TCCgroup}, and we use it repeatedly.
\begin{prop}\label{handy}
Let $k=k_s$. Let $G$ be connected. Suppose that a (possibly non-$k$-defined) subgroup $H$ of $G$ is not $G$-cr over $k$. If a $k$-subgroup $N$ of $G$ normalizes $H$, then there exist a proper $k$-parabolic subgroup of $G$ containing $H$ and $N$. 
\end{prop}
\begin{proof}
By~\cite[Prop.~3.3]{Uchiyama-Nonperfect-pre}, there exists a proper $k$-parabolic subgroup $P$ containing $H$ and $N(k)$. Since the $k_s$-points are dense in $N$ by~\cite[AG.13.3]{Borel-AG-book}, $P$ contains $H$ and $N$. 
\end{proof}

\section{Complete reducibility and strong reductivity}
In this paper, we extend various other results concerning complete reducibility in~\cite{Bate-geometric-Inventione} to an arbitrary $k$. First, we extend the notion of \emph{strong reductivity}~\cite[Def.~16.1]{Richardson-conjugacy-Duke}.

\begin{defn}\label{strong}
Let $k=k_s$. Let $H$ be a subgroup of $G$. Then $H$ is \emph{strongly reductive over $k$ in $G$} if $H$ is not contained in any proper $k$-defined $R$-parabolic subgroup of the reductive $k$-group $C_G(S)$, where $S$ is a maximal $k$-torus of $C_G(H)$.
\end{defn} 
Note that this definition does not depend on the choice of $S$. Also note that Definitions~\ref{strong} (and~\ref{def-regular} below) make sense even if $H$ is not $k$-defined. We generalize~\cite[Thm.~3.1]{Bate-geometric-Inventione}, which was the main result of~\cite{Bate-geometric-Inventione}. 
\begin{thm}\label{stronglyreductive}
Let $k=k_s$. Let $H$ be a (possibly non-$k$-defined) subgroup of $G$. Then $H$ is $G$-cr over $k$ if and only if $H$ is strongly reductive over $k$ in $G$.
\end{thm}

\begin{proof}
Suppose that $H$ is $G$-cr over $k$. Let $S$ be a maximal $k$-defined torus of $C_G(H)$. Suppose that $S$ is central in $G$. Then $C_G(S)=G$. Suppose that $H$ is contained in a proper $k$-defined $R$-parabolic subgroup $P$ of $G$. Then there exists a $k$-defined $R$-Levi subgroup $L$ of $P$ containing $H$ since $H$ is $G$-cr over $k$. Since $k=k_s$, we can set $L=L_{\lambda}$ and $P=P_{\lambda}$ for some $\lambda\in Y_{k}(G)$. Then $\lambda(\overline k^*)<C_G(H)$ and $\lambda(\overline k^*)$ is a connected commutative non-central $k$-subgroup of $G$. Now let $C:=\overline{C_{G}(H)(k_s)}$. Then $C$ is the unique maximal $k$-defined subgroup of $C_{G}(H)$. Since $k=k_s$, maximal $k$-tori of $C$ are $G(k)$-conjugate by~\cite[Thm.~20.9]{Borel-AG-book}. Then $\lambda(\overline k^*)<S$ since $S$ is central in $G$. This is a contradiction. So $H$ cannot be contained in a proper $k$-defined $R$-parabolic subgroup of $G$. Therefore $H$ is strongly reductive over $k$.   

Now we assume $S$ is non-central in $G$. Suppose that $H$ is contained in a $k$-defined proper $R$-parabolic subgroup $Q$ of $C_G(S)$. Note that $C_G(S)$ is a $k$-defined $R$-Levi subgroup of $G$ (\cite[Cor.~6.10]{Bate-geometric-Inventione}). Then by the rational version of~\cite[Lem.~6.2(\rmnum{2})]{Bate-geometric-Inventione} (note that \cite[Lem.~6.2(\rmnum{2})]{Bate-geometric-Inventione} is for $k=\overline k$, but the same proof works work by word if we set $P=P_{\lambda}$, $P'=P'_{\mu}$, and $L=L_{\lambda}$ for $\lambda\in Y_k(G)$ in the proof), there exists a $k$-defined proper $R$-parabolic subgroup $P_\mu$ of $G$ such that $Q=C_G(S)\cap P_\mu$. It is clear that $S<Q< P_\mu$. Since $H$ is $G$-cr over $k$, there exists a $k$-defined $R$-Levi subgroup $L$ of $P_\mu$ containing $H$. Without loss we set $L=L_\mu$. Then $\mu(\overline k^{*})$ is a $k$-torus in $C_{P_{\mu}}(H)$. Since $S$ is contained in $P_{\mu}$ and $S$ is a maximal $k$-torus of $C_G(H)$, $S$ is a maximal $k$-torus of $C_{P_\mu}(H)$. Since $k=k_s$, by the same argument as in the first paragraph, we have $g\mu(\overline k^*)g^{-1}<S$ for some $g\in P_{\mu}(k)$. Then $C_G(S)< C_G(g\mu(\overline{k}^{*})g^{-1})= gL_{\mu}g^{-1} < P_{\mu}$. Therefore $Q=C_G(S)$, which is a contradiction.

Now suppose that $H$ is strongly reductive over $k$. Let $S$ be a maximal $k$-torus of $C_G(H)$. Then $H$ is not contained in any proper $k$-defined $R$-parabolic subgroup of $C_G(S)$. Let $L:=C_G(S)$. Let $Q$ be a $k$-defined $R$-parabolic subgroup of $G$ containing $L$ as a $k$-Levi subgroup. Then by the rational version of~ \cite[Lem.~6.2(\rmnum{2})]{Bate-geometric-Inventione}, $Q$ is minimal among all $k$-defined $R$-parabolic subgroups of $G$ containing $H$. Let $P$ be a $k$-defined $R$-parabolic subgroup of $G$ containing $H$. Our goal is to find a $k$-defined $R$-Levi subgroup of $P$ containing $H$. 

If $P'$ is a $k$-defined $R$-parabolic subgroup of $G$ such that $P'<P$ and $M'$ is a $k$-defined $R$-Levi subgroup of $P'$, then, by~\cite[Cor.~6.6]{Bate-geometric-Inventione}, there exists a unique $\overline k$-defined $R$-Levi subgroup $M''$ of $P$ containing $M'$. But $M''$ is $k$-defined by~\cite[Lem.~2.5(\rmnum{3})]{Bate-uniform-TransAMS}. So, we assume that $P$ is minimal among all $k$-defined $R$-parabolic subgroups of $G$ containing $H$. By~\cite[Prop.~20.7]{Borel-AG-book}, $P\cap Q$ contains a maximal $k$-torus $T$ of $G$. 
We see that there exists a (possibly non-$k$-defined) common $R$-Levi subgroup $M$ of $P$ and $Q$ containing $T$ by a similar argument to that in the proof of~\cite[Thm.~3.1]{Bate-geometric-Inventione} (use~\cite[Lem.~6.2]{Bate-geometric-Inventione} where necessary). Since $T$ is $k$-defined, $M$ is $k$-defined by~\cite[Lem.~2.5(\rmnum{3})]{Bate-uniform-TransAMS}.

Let $P^{-}$ be the unique opposite of $P$ such that $M=P\cap P^{-}$. By the argument in the third paragraph of the proof of~\cite[Thm.~3.1]{Bate-geometric-Inventione}, we have
\begin{equation*}
R_u(Q)=(R_u(Q)\cap M)(R_u(Q)\cap R_u(P^{-}))(R_u(Q)\cap R_u(P)).
\end{equation*}
It is clear that $R_u(Q)\cap M$ is trivial. Since $L$ and $M$ are $k$-defined $R$-Levi subgroups of $Q$, there exists $u'\in R_u(Q)(k)$ such that $u'Mu'^{-1}=L$ by~\cite[Lem.~2.5(\rmnum{3})]{Bate-uniform-TransAMS}. Using the rational version of the Bruhat decomposition (\cite[Thm.~21.15]{Borel-AG-book}), we can express $u'$ as $u'=yz$ where $y\in (R_u(Q)\cap R_u(P^{-}))(k)$ and $z\in (R_u(Q)\cap R_u(P))(k)$. Note that $zMz^{-1}$ is a $k$-defined common $R$-Levi subgroup of $P$ and $Q$ because $z\in (R_u(Q)\cap R_u(P))(k)$. So, without loss, we may assume $z=1$. Then $y M y^{-1} = L$ and $L<P^{-}$. Since $L$ contains $H$, we have $H < P\cap P^{-}=M$. We are done.

\end{proof}

\begin{rem}
This result shows that Richardson's various results on strongly reductivity in~\cite{Richardson-conjugacy-Duke} and~\cite{Richardson-orbits-BullAustralian} may be extended to an arbitrary (or a separably closed) $k$ using methods and results in this paper. 
\end{rem}

\section{Complete reducibility and regular subgroups}
Generalizing the notion of a \emph{regular subgroup} of $G$~\cite{Liebeck-Seitz-memoir},~\cite{Liebeck-Seitz-Steinberg-JAlgebra}, define:
\begin{defn}\label{def-regular}
A (possibly non-$k$-defined) subgroup $H$ of $G$ is \emph{$k$-regular} if $H$ is normalized by a maximal $k$-torus of $G$. 
\end{defn}
We extend~\cite[Prop.~3.19]{Bate-geometric-Inventione}.
\begin{lem}\label{maximaltoruscentralizer}
Let $k=k_s$. Let $H$ be a reductive $k$-subgroup of $G$. Let $K$ be a (possibly non-$k$-defined) subgroup of $H$. Suppose that $H$ contains a maximal $k$-torus of $C_G(K)$ and that $K$ is $G$-cr over $k$. Then $K$ is $H$-cr over $k$ and $H$ is $G$-cr over $k$.
\end{lem}
\begin{proof}
Let $S$ be a maximal $k$-torus of $C_G(K)$ contained in $H$. Then $S$ is a maximal $k$-torus of $C_H(K)$. Since $K$ is $G$-cr over $k$, $K$ is $C_G(S)$-ir over $k$ by~Theorem~\ref{stronglyreductive}. Note that $K<C_H(S)<C_G(S)$. So $K$ is $C_H(S)$-ir over $k$. Thus $K$ is $H$-cr over $k$ by~Theorem~\ref{stronglyreductive}. 

Let $P$ be a $k$-defined $R$-parabolic subgroup of $G$ containing $H$. Then $P$ contains $K$. Since $K$ is $G$-cr over $k$, by the same argument as in the second paragraph of the proof of Theorem~\ref{stronglyreductive}, there exists $\lambda\in Y_k(G)$ such that $C_G(S)<L_\lambda$ and $P=P_\lambda$. Thus we have $\lambda(\overline k^*)<C_G(C_G(S))^{\circ}=Z(C_G(S))^{\circ}$. It is clear that $S<Z(C_G(S))^{\circ}$. We have $Z(C_G(S))^{\circ}<C_G(K)$. By~\cite[Thm.~18.2]{Borel-AG-book}, $Z(C_G(S))^{\circ}$ is $k$-defined. Since $S$ is a maximal $k$-torus of $C_G(K)$, we have $S=Z(C_G(S))^{\circ}$. Thus $\lambda(\overline k^*)<S<H$. So $\lambda\in Y_k(H)$. Thus we have $P_\lambda(H)=P_\lambda\cap H=H$. So $\lambda\in Y_k(Z(H))$. Then $H<C_G(\lambda(\overline k^*))=L_\lambda$, and we are done.
\end{proof}
The following the main result in this section. We extend~\cite[Prop.~3.20]{Bate-geometric-Inventione}.
\begin{thm}\label{regular}
Let $k=k_s$. Let $G$ be connected. Let $H$ be a $k$-regular reductive $k$-subgroup of $G$. Then $H$ is $G$-cr over $k$.
\end{thm}
\begin{proof}
Let $T$ be a maximal $k$-torus of $G$ normalizing $H$. Then if we show that $TH$ is $G$-cr over $k$, by~\cite[Prop.~3.5]{Uchiyama-Nonperfect-pre} we are done since $H$ is a normal subgroup of the $k$-group $TH$. Note that $C_G(T)=T$ and $T$ is $G$-cr over $k$ by~\cite[Cor.~9.8]{Bate-cocharacter-Arx}. Applying Lemma~\ref{maximaltoruscentralizer} to $T<TH<G$, we obtain the desired result.
\end{proof}

\section{Complete reducibility and linear reductivity}
In this section we assume that $G$ is connected. Recall that a subgroup $H$ of $G$ is called \emph{linearly reductive} if every rational representation of $H$ is completely reducible. It is known that a linearly reductive subgroup $H$ of $G$ is $G$-cr~\cite[Lem.~2.6]{Bate-geometric-Inventione} and if $H$ is $k$-defined, it is $G$-cr over $k$~\cite[Cor.~9.8]{Bate-cocharacter-Arx}. 

It is clear that the converse of Proposition~\ref{algclocentralizer} is false; take $H$ to be a Borel subgroup of $G$. However we have the following partial converse~\cite[Cor.~3.18]{Bate-geometric-Inventione}:
\begin{lem}\label{lem1}
Let $k=\overline k$. Let $H$ be a subgroup of $G$. If $C_G(H)$ is $G$-ir, then $H$ is linearly reductive. In particular, $H$ is $G$-cr.
\end{lem}
The previous lemma was a consequence of the following~\cite[Lem.~3.38]{Bate-geometric-Inventione}:
\begin{lem}\label{lem2}
Let $k=\overline k$. Let $H$ be a subgroup of $G$. If $H$ is $G$-ir, then $C_G(H)$ is linearly reductive. 
\end{lem}

\begin{rem}\label{importantexample}
A natural analogue of Lemmas~\ref{lem1} and~\ref{lem2} for an arbitrary field $k$ is false even if $H$ is $k$-defined. Let $k$ be a nonperfect field of characteristic $2$. Let $a\in k\backslash k^2$. Let $G=PGL_2$. We write $\overline{A}$ for the image in $PGL_2$ of $A\in GL_2$. Let $H:=\left\{\overline{\left[\begin{array}{cc}
                                                      x & ay \\
                                                      y & x \\
                                                    \end{array}  
                                             \right]} \in PGL_2(\overline k) \mid x, y \in \overline k\right\}$.
Then $H$ is $G$-ir over $k$ since $H$ contains the element $u:=\overline{\left[\begin{array}{cc}
                                                      0 & a \\
                                                      1 & 0 \\
                                                    \end{array}  
                                             \right]}$, which is a $k$-anisotropic unipotent element; see~\cite[Ex.~3.8]{Uchiyama-Nonperfect-pre}. It is clear that $C_G(H)=H$ is not a torus. So, by~\cite[Thm.~2]{Nagata-complete-Kyoto}, $C_G(H)=H$ is not linearly reductive. We see that $H$ is not $G$-cr since $H$ is unipotent. 
\end{rem}

\begin{defn}
A unipotent element $u\in G(k)$ is called \emph{$k$-plongeable} if $u$ belongs to the unipotent radical of some proper $k$-parabolic subgroup of $G$. 
\end{defn}

Note that a unipotent element $u$ in Remark~\ref{importantexample} is $k$-nonplongeable. For more on $k$-nonplongeable unipotent elements, see~\cite{Gille-E8unipotent-QJM},~\cite{Tits-unipotent-Utrecht},~\cite{Tits-Note1-CollegeFrance}. Recall the following deep result which was conjectured by Tits~\cite{Tits-Note1-CollegeFrance} and proved by Gille~\cite{Gille-unipotent-Duke}:

\begin{prop}\label{GillePlong}
Let $k=k_s$. Let $G$ be a semisimple simply connected $k$-group. If $[k:k^p]\leq p$, then every unipotent subgroup of $G(k)$ is $k$-plongeable.
\end{prop} 

The following are partial extensions of Lemmas~\ref{lem2} and~\ref{lem1} to a separably closed $k$ under the $k$-plongeability hypothesis (see Remark~\ref{lin} below).

\begin{prop}\label{linred}
Let $k=k_s$. Suppose that a $k$-subgroup $H$ of $G$ is $G$-ir over $k$. Let $C:=\overline{C_G(H)(k)}$ (or $C_G(H)$). If every unipotent element of $G(k)$ is $k$-plongeable (in particular if $G$ and $k$ satisfy the hypothesis of Proposition~\ref{GillePlong}), then every element of $C(k)$ is semisimple.
\end{prop}
\begin{proof}
Let $C:=\overline{C_G(H)(k)}$. Suppose that there exists a non-trivial unipotent element $u\in C(k)$. Since we assumed that every unipotent element of $G(k)$ (hence every element of $C(k)$) is $k$-plongeable, there exists a proper $k$-parabolic subgroup $P$ of $G$ such that $u\in R_u(P)$. Then, it is clear that the subgroup $U:=\langle u \rangle$ is not $G$-cr over $k$. Since $H$ normalizes $U$ and $H$ is $k$-defined, by Proposition~\ref{handy}, there exists a proper $k$-parabolic subgroup $P'$ of $G$ containing $U$ and $H$. This is a contradiction. Therefore every element of $C(k)$ is semisimple. The same argument works for $C:=C_G(H)$. 
\end{proof}      
\begin{prop}\label{linred2}
Let $k=k_s$. Let $H$ be a subgroup of $G$. Suppose that $C:=\overline{C_G(H)(k)}$ is $G$-ir over $k$. If every unipotent element of $H(k)$ is $k$-plongeable, then every element of $H(k)$ is semisimple. 
\end{prop}
\begin{proof}
Swap the roles of $H$ and $\overline{C_G(H)(k)}$ in the proof of Proposition~\ref{linred}.
\end{proof}
\begin{rem}\label{lin}
We do not know whether $C$ in Proposition~\ref{linred} (or $H$ in Proposition~\ref{linred2}) is linearly reductive. For that purpose we need to know whether every element of $C(\overline k)$ (or $H(\overline k)$) is semisimple~\cite[Thm.~2]{Nagata-complete-Kyoto}. 
\end{rem}

Recall that~\cite[Def.~3.27]{Bate-geometric-Inventione}, a subgroup $H$ of $G$ is called \emph{separable} if the scheme theoretic centralizer of $H$ in $G$ (in the sense of~\cite[Def.~A.1.9]{Conrad-pred-book}) is smooth. It is known that every subgroup of $GL_n$ is separable~\cite[Ex.~3.28]{Bate-geometric-Inventione}. The following is a generalization of~\cite[Lem.~2.6 and Cor.~3.17]{Bate-geometric-Inventione}.

\begin{prop}
Suppose that a $k$-subgroup $H$ of $G$ is linearly reductive. Then $C_G(H)$ is $k$-defined and $G$-cr over $k$. 
\end{prop}
\begin{proof}
Since $H$ is linearly reductive, it is $G$-cr~\cite[Lem.~2.6]{Bate-geometric-Inventione}. So $C_G(H)$ is reductive by~\cite[Prop.~3.12]{Bate-geometric-Inventione}. Then $C_G(H)$ is $G$-cr over $k$ by Theorem~\ref{main}. Since $H$ is linearly reductive, $H$ is separable in $G$. So, by the proof of~\cite[Prop.~7.4]{Bate-cocharacter-Arx}, $C_G(H)$ is $k$-defined since $H$ is $k$-defined.  
\end{proof}

\section{Centralizers of completely reducible subgroups}
\begin{proof}[Proof of Theorem~\ref{main}]
We start with the first part of the theorem. Let $C:=\overline{C_G(H)(k)}$. Let $P$ be a $k$-parabolic subgroup of $G$ such that $HC < P$. Since $H$ is $G$-cr over $k$, there exists a $k$-Levi subgroup $L$ of $P$ with $H < L$. Let $\lambda$ be a $k$-cocharacter of $G$ such that $P=P_\lambda$ and $L=L_\lambda$. Then $\lambda$ is a cocharacter of $C$. We have $C=(C\cap L)(C \cap R_u(P))$. Since $\lambda$ normalizes $C$, by~\cite[Prop.~2.2]{Bate-cocharacter-Arx} $C\cap R_u(P)$ is $k$-defined. For $u\in (C \cap R_u(P))(k)$, we define a $k$-morphism $\phi_u: \overline{k}\rightarrow C\cap R_u(P)$ by $\phi_u(0)=1$ and $\phi_u(t)=\lambda(t)u\lambda(t)^{-1}$ for $t\in \overline{k}^*$. Then the image of $\phi_u$ is a connected $k$-subvariety of $C\cap R_u(P)$ containing $1$ and $\phi_u(1)=u$. Then $C\cap R_u(P)$ must be trivial since $C$ is pseudo-reductive. Thus $C< L$. Therefore $HC$ is $G$-cr over $k$. Note that $C$ is a normal subgroup of the $k$-defined subgroup $HC$. So by~\cite[Prop.~3.5]{Uchiyama-Nonperfect-pre}, $C$ is $G$-cr over $k$.  

For the second part, the same argument shows that $HC_G(H)$ is $G$-cr over $k$ since we assumed that $C_G(H)$ is reductive. However, if $C_G(H)$ is not $k$-defined we cannot apply~\cite[Prop.~3.5]{Uchiyama-Nonperfect-pre} to conclude that $C_G(H)$ is $G$-cr over $k$. We need a different argument. Let $P'$ be a minimal $k$-parabolic  subgroup of $G$ containing $H C_G(H)$. Since $H C_G(H)$ is $G$-cr over $k$, there exists a $k$-Levi subgroup $L'$ of $P'$ containing $H C_G(H)$. If $C_G(H)$ is $L'$-cr over $k$, it is $G$-cr over $k$ by~\cite[Lem.~3.4]{Uchiyama-Nonperfect-pre}. Otherwise, by~\cite[Lem.~3.4]{Uchiyama-Nonperfect-pre} and Proposition~\ref{handy}, there exist a proper $k$-parabolic subgroup $P_{L'}$ of $L'$ containing $C_G(H)$ and $H$ since $H$ is $k$-defined and $H$ normalizes $C_G(H)$. Note that $P_{L'}\ltimes R_u(P')$ is a $k$-parabolic subgroup of $G$ by~\cite[Prop.~4.4(c)]{Borel-Tits-Groupes-reductifs} and it is properly contained in $P'$. This contradicts the minimality of $P'$.  
\end{proof}

The following are further consequences of Theorem~\ref{TCCgroup}, and they all deal with special cases of Open Problem~\ref{centralizerquestion}.
\begin{prop}\label{unipotent}
Let $k=k_s$. Let $G$ be connected. If a $k$-subgroup $H$ of $G$ is $G$-cr over $k$ and if $\overline{C_G(H)(k)}^{\circ}$ (or ${C_G(H)}^{\circ}$) is unipotent, then $\overline{C_G(H)(k)}$ (or $C_G(H)$) is $G$-cr over $k$. 
\end{prop}
\begin{proof}
Let $C:=C_G(H)$. Suppose that $C$ is not $G$-cr over $k$. By Proposition~\ref{handy} there exists a $k$-defined proper parabolic subgroup $P_\lambda$ of $G$ containing $H$ and $C$. Then there exists a $k$-Levi subgroup $L$ of $P_\lambda$ containing $H$ since $H$ is $G$-cr over $k$. Without loss, we assume $L=L_\lambda$. Then $\lambda(\overline k^*)<C^{\circ}$ since  $\lambda(\overline k^*)$ is connected. So $\lambda$ must be trivial since $C^{\circ}$ is unipotent. This is a contradiction. The other case can be shown in the same way.
\end{proof}

\begin{rem}
See~\cite[Sec.~4,5]{Uchiyama-Nonperfect-pre} for examples of a $k$-subgroup $H$ of connected $G$ (or non-connected $G$) such that: 1.~$H$ is $G$-cr over $k$, 2.~$C_G(H)$ (or $\overline{C_G(H)(k_s)}$) is unipotent.
\end{rem}

\begin{cor}\label{cent}
Let $k=k_s$. Let $G$ be connected. Let $H$ be a $k$-subgroup of $G$. If $H$ is $G$-ir over $k$, then $\overline{C_G(H)(k)}$ is $G$-cr over $k$.
\end{cor}
\begin{proof}
By~Proposition~\ref{handy}, there exists a proper $k$-parabolic subgroup of $G$ containing $H$ and $\overline{C_G(H)(k)}$. This is a contradiction since $H$ is $G$-ir over $k$.
\end{proof}

\begin{cor}\label{centralizer1}
Let $k=k_s$. Let $G$ be connected. Let $H$ be a (possibly non-$k$-defined) subgroup of $G$. If $\overline{C_G(H)(k)}$ is $G$-ir over $k$, then $H$ is $G$-cr over $k$.
\end{cor}
\begin{proof}
The same proof as that of Corollary~\ref{cent} works. 
\end{proof}

\begin{rem}
In Corollary~\ref{centralizer1}, we cannot replace $\overline{C_G(H)(k)}$ by $C_G(H)$ even if $H$ is $k$-defined; if $C_G(H)$ is not $k$-defined, there might not be any proper $k$-parabolic subgroup containing $H$ and $C_G(H)$. It would be interesting to know whether such examples exist. 
\end{rem} 
By a similar argument to that in the proof of Corollary~\ref{cent}, we obtain:

\begin{cor}\label{normalizer-of-Gcrsubgroup}
Let $k=k_s$. Let $G$ be connected. If a $k$-subgroup $H$ of $G$ is $G$-ir over $k$, then $N_G(H)$ and $\overline{N_G(H)(k)}$ are $G$-cr over $k$.
\end{cor}

\begin{cor}
Let $k=k_s$. Let $G$ be connected. Let $H$ be a (possibly non-$k$-defined) subgroup of $G$. If $\overline{N_G(H)(k)}$ is $G$-ir over $k$, then $H$ is $G$-cr over $k$.
\end{cor}
\noindent It is natural to ask:
\begin{oprob}\label{normalizerquestion}
Let $k$ be a field. Suppose that a $k$-subgroup $H$ of $G$ is $G$-cr over $k$. Is $N_G(H)$ $G$-cr over $k$? Is $\overline{N_G(H)(k_s)}$ $G$-cr over $k$? 
\end{oprob}

Propositions~\ref{prop1} and~\ref{prop2} below show that if we allow $H$ to be non-$k$-defined, the answer to Open Problem~\ref{centralizerquestion} is no. First we need~\cite[Prop.~7.4]{Bate-cocharacter-Arx}

\begin{lem}\label{centralizerkdefined}
Let $k$ be a field. If a $k$-subgroup $H$ of $G$ is separable in $G$, then $C_G(H)^{\circ}$ is $k$-defined. 
\end{lem}

\begin{prop}\label{prop1}
Let $k$ be a nonperfect field of characteristic $3$. Let $a\in k\backslash k^3$. Let $G=GL_4$. Then there exists a $\overline k$-defined subgroup $H$ of $G$ such that $H$ is $G$-cr over $k$ but $C_G(H)$ is not $G$-cr over $k$.
\end{prop}
\begin{proof}
Let $h_1= \left[\begin{array}{cccc}
                        0 & 0 & a & 0 \\
                        1 & 0 & 0 & 0 \\
                        0 & 1 & 0 & 0 \\
                        0 & 0 & 0 & a^{1/3} \\
                         \end{array}  
                        \right]$,
     $h_2= \left[\begin{array}{cccc}
                        1 & a^{1/3} & 2a^{2/3} & 0 \\
                        0 & 1 & 0 & 0 \\
                        0 & 0 & 1 & 0 \\
                        0 & 0 & 0 & 1 \\
                         \end{array}  
                        \right]$. Set $H:=\langle h_1, h_2 \rangle$. \\

\noindent A simple matrix computation shows
$C_G(H)= \left\{\left[\begin{array}{cccc}
                        s & 0 & 0 & x_1 \\
                        0 & s & 0 & a^{-1/3}x_1 \\
                        0 & 0 & s & a^{-2/3}x_1 \\
                        0 & 0 & 0 & t \\
                         \end{array}  
                        \right]\in GL_4(\overline{k}) \middle| x_1, s, t \in \overline{k}\right\}$.
Note that a subgroup $H$ of $G=GL_n(V)$ is $G$-cr over $k$ if and only if $H$ acts $k$-semisimply on $V$~\cite[Ex.~3.2.2(a)]{Serre-building}. We find by a direct computation that $H$ acts semisimply on a $4$-dimensional $k$-vector space in the usual way with $k$-irreducible summands $V_1:=\left[\begin{array}{c}
                        * \\
                        * \\
                        * \\
                        0 \\
                         \end{array}  
                        \right]$ and $V_2:=\left[\begin{array}{c}
                        0 \\
                        0 \\
                        0 \\
                        * \\
                         \end{array}  
                        \right]$. Hence $H$ is $G$-cr over $k$. It is clear that $C_G(H)$ is not $G$-cr over $k$ since $V_1$ is a $k$-defined $3$-dimensional $C_G(H)$-stable subspace with no $C_G(H)$ stable complement. 

We show that $H$ is not $k$-defined. First, we see that $C_G(H)^{\circ}(k_s)$ is not dense in $C_G(H)^{\circ}$, so $C_G(H)^{\circ}$ is not $k$-defined by~\cite[AG.~13.3]{Borel-AG-book}. We conclude that by Lemma~\ref{centralizerkdefined}, $H$ cannot be $k$-defined since $H$ is separable in $G$.
\end{proof}

\begin{prop}\label{prop2}
Let $k$ be a nonperfect field of characteristic $p$. Let $a\in k\backslash k^p$. Let $G=GL_4$. Then there exists a $\overline k$-defined subgroup $H$ of $G$ such that $H$ is $G$-cr over $k$, but $\overline{C_G(H)(k_s)}$ is not $G$-cr over $k$.
\end{prop}
\begin{proof}
Let $h_1= \left[\begin{array}{cccc}
                        1 & 1 & 1 & 0 \\
                        0 & 1 & 0 & 0 \\
                        0 & 0 & 1 & 0 \\
                        0 & 0 & 0 & 1 \\
                         \end{array}  
                        \right]$,
     $h_2= \left[\begin{array}{cccc}
                        1 & 0 & 0 & 0 \\
                        a & 1 & 0 & 0 \\
                        a^{1/p} & 0 & 1 & 0 \\
                        0 & 0 & 0 & 1 \\
                         \end{array}  
                        \right]$. Define
     $H:=\langle h_1, h_2 \rangle$. Then

\noindent$C_G(H)= \left\{\left[\begin{array}{cccc}
                        s & 0 & 0 & 0 \\
                        0 & t & a^{\frac{p-1}{p}}(s-t) & x_2 \\
                        0 & s-t & (1-a^{\frac{p-1}{p}})s+a^{\frac{p-1}{p}}t & -x_2\\
                        0 & y_2 & a^{\frac{p-1}{p}}y_2 & w \\
                         \end{array}  
                        \right]\in GL_4(\overline{k}) \middle| x_2, y_2, s, t, w \in \overline{k}\right\}$. \\
So, $\overline{C_G(H)(k_s)}= \left\{\left[\begin{array}{cccc}
                        s & 0 & 0 & 0 \\
                        0 & s & 0 & x_2 \\
                        0 & 0 & s & -x_2\\
                        0 & 0 & 0 & w \\
                         \end{array}  
                        \right]\in GL_4(\overline k) \middle| x_2, s, w \in \overline{k}\right\}$.
A similar argument to that in the proof of Proposition~\ref{prop1} shows that $H$ is $G$-cr over $k$ and $H$ is not $k$-defined. It is clear that $\overline{C_G(H)(k_s)}$ is not $G$-cr over $k$.
\end{proof}

\section{On the structure of the set of $R$-parabolic subgroups}

\begin{proof}[Proof of Theorem~\ref{TCCnotsubset}]
Let $\tilde G=SL_3$. Let $G=\tilde G \rtimes \langle \sigma \rangle$ where $\sigma$ is the nontrivial graph automorphism of $\tilde G$. Fix a $k$-split maximal torus $T$, and a $k$-Borel subgroup $B$ of $G$ containing $T$. Let $\alpha, \beta$ be the simple roots of $G$ corresponding to $T$ and $B$. Let $n_\alpha, n_\beta$ be the canonical reflections corresponding to $\alpha$ and $\beta$ respectively. Let $f$ be the automorphism of $G$ such that $f(A)=(A^{T})^{-1}$. Then we obtain 
\begin{equation}
\sigma(P_i) = (n_\alpha n_\beta n_\alpha) f(P_i) (n_\alpha n_\beta n_\alpha)^{-1} \text{ for }i\in \{\alpha, \beta\}. \label{contrans}
\end{equation}  
In particular, we have
$
\sigma(P_\alpha)=P_\beta,\; \sigma(P_\beta)=P_\alpha.
$
Let $W_k$ be the Weyl group of $G^{\circ}$. Then $W_k\cong S_3$. We list all canonical representatives of $W_k$:
$
1, n_\alpha, n_\beta, n_\alpha n_\beta, n_\beta n_\alpha, n_\alpha n_\beta n_\alpha.
$
Taking all $W_k$-conjugates of $P_\alpha$ and $P_\beta$, we obtain all proper maximal $k$-parabolic subgroups of $G^{\circ}$ containing $T$: $P_\alpha$, $P_\beta$, $P_{-\alpha}$, $P_{-\beta}$, $n_\beta\cdot P_{\alpha}$, $n_\alpha\cdot P_\beta$. Using (\ref{contrans}), we find that none of these proper maximal $k$-parabolic subgroups of $G^{\circ}$ is normalized by $\sigma$. So, by~\cite[Prop.~6.1]{Bate-geometric-Inventione} they are $k$-defined proper maximal $R$-parabolic subgroups of $G$ containing $T$.

Now we look at $k$-Borel subgroups of $G^{\circ}$. By taking all $W_k$-conjugates of $B$, we obtain all $k$-Borel subgroups of $G^{\circ}$ containing $T$: $B$, $n_\alpha\cdot B$, $n_\beta B$, $n_\alpha n_\beta\cdot B$, $n_\beta n_\alpha\cdot B$, $n_\alpha n_\beta n_\alpha \cdot B$. 

\begin{lem}\label{A2}
$C_{Y_k(T)}(\sigma)=a(\alpha^{\vee}+\beta^{\vee}), \textup{ where } a\in \mathbb{Z}$.
\end{lem}
\begin{proof}
Let $\lambda = x\alpha^{\vee}+y\beta^{\vee}$. Using (\ref{contrans}), we obtain 
$
\sigma(\lambda)=n_\alpha n_\beta n_\alpha \cdot (-x\alpha^{\vee}-y\beta^{\vee}) 
                        =y\alpha^{\vee}+x\beta^{\vee}$. 
Then, for $\lambda\in C_{Y_k(T)}(\sigma)$ we must have $x=y$.
\end{proof}

\begin{lem}\label{H1-lemma}
$H:=\langle \sigma, B\rangle$ is a $k$-defined $R$-parabolic subgroup of $G$.
\end{lem}
\begin{proof}
Let $\lambda = \alpha^{\vee}+\beta^{\vee}$. By Lemma~\ref{A2}, we have $\sigma \in L_\lambda$. An easy calculation shows that $P_\lambda= H$. 
\end{proof}

It is clear that $H$ is not $G$-cr over $k$. Note that $\Lambda(G)^{H}=\{ H \}$, and $H$ is not a simplex in $\Delta(G)=\Delta(G^0)$. This gives the first part of the theorem. Consider the set of $k$-defined $R$-parabolic subgroups of $G$ containing $T$. We have
\begin{equation*}
B<B, B<P_\alpha, B<P_\beta, B<P_\lambda, B<G.
\end{equation*}
Thus the cardinality of the set of $R$-parabolic subgroups of $G$ containing $B$ is $5$, which is not a power of $2$. So $\Lambda(G)$ ordered by reverse inclusion is not a simplicial complex (in the sense of~\cite[Thm.~5.2]{Tits-book}) since it cannot be isomorphic to the partially ordered set of subsets of some finite set; see Figure 1 where vertices (edges) correspond to $k$-defined maximal (minimal) $R$-parabolic subgroups of $G$ containing $T$. 
\end{proof}

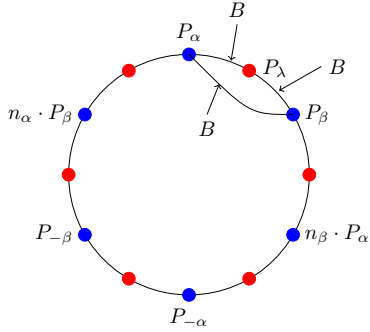
\begin{figure}[!h]
\begin{center}
\scalebox{0.8}{
\begin{tikzpicture}
\draw (0,0) circle (2cm);
\filldraw [blue] (0,2) circle (3pt);
\draw [above] (0,2.1) node {$P_\alpha$};
\filldraw [blue] (0,-2) circle (3pt);
\draw [below] (0,-2.1) node {$P_{-\alpha}$};
\filldraw [blue] (1.73, 1) circle (3pt);
\draw [right] (1.80, 1) node {$P_{\beta}$};
\filldraw [blue] (1.73, -1) circle (3pt);
\draw [right] (1.80, -1) node {$n_\beta\cdot P_{\alpha}$};
\filldraw [blue] (-1.73, 1) circle (3pt);
\draw [left] (-1.80, -1) node {$P_{-\beta}$};
\filldraw [blue] (-1.73, -1) circle (3pt);
\draw [left] (-1.80, 1) node {$n_\alpha\cdot P_{\beta}$};
\filldraw [red] (1, 1.73) circle (3pt);
\draw [right] (1.1, 1.77) node {$P_\lambda$};
\filldraw [red] (1, -1.73) circle (3pt);
\filldraw [red] (2, 0) circle (3pt);
\filldraw [red] (-2, 0) circle (3pt);
\filldraw [red] (-1, 1.73) circle (3pt);
\filldraw [red] (-1, -1.73) circle (3pt);
\draw (0,2) .. controls (1,1) .. (1.73, 1);
\draw [->] (0.8, 2.5)--(0.7, 1.9); 
\draw [above] (0.8, 2.5) node {$B$};
\draw [->] (2.2, 1.8)--(1.5, 1.4); 
\draw [right] (2.2, 1.8) node {$B$};
\draw [->] (0.3, 1.0)--(0.5, 1.5); 
\draw [below] (0.3, 1.0) node {$B$};
\end{tikzpicture}
}
\caption{The set of $R$-parabolic subgroups of $G=SL_3$ containing $T$}
\end{center}
\end{figure}

\begin{oprob}
Let $k$ be a field. Let $G$ be non-connected. Suppose that a (possibly non-$k$-defined) subgroup $H$ of $G$ is not $G$-cr over $k$. If a $k$-subgroup $N$ of $G$ normalizes $H$, does there exist a $k$-defined proper $R$-parabolic subgroup of $G$ containing $H$ and $N$?
\end{oprob}

\section{The number of overgroups of $G$-ir subgroups}
Recall that Proposition~\ref{overgroups} depended on the following~\cite[Lem.~2.1]{Liebeck-Testerman-irreducible-QJM}:
\begin{lem}\label{finiteness}
Let $k$, $G$, $H$ be as in the hypotheses of Proposition~\ref{overgroups}. Then $C_G(H)$ is finite.
\end{lem}
However if $k$ is nonperfect, we have 
\begin{prop}
Let $k$ be a nonperfect field of characteristic $2$. Let $G=PGL_2$. Then there exists a connected $k$-subgroup $H$ of $G$ such that $H$ is $G$-ir over $k$ but $C_G(H)$ is infinite.
\end{prop}
\begin{proof}
Let $a\in k\backslash k^2$. Let $H:=\left\{\overline{\left[\begin{array}{cc}
                                                      x & ay \\
                                                      y & x \\
                                                    \end{array}  
                                             \right]} \in PGL_2(\overline{k}) \mid x, y \in \overline{k}\right\}$. Then $H$ is connected and $G$-ir over $k$; see Remark~\ref{importantexample}. We have $H=C_G(H)$, and $C_G(H)$ is infinite.
\end{proof}

\begin{proof}[Proof of Theorem~\ref{infinitelymanyconjugacy}]
Let $a\in k\backslash k^2$. Define \\$H:= \left\{\overline{\left[\begin{array}{cccc}
                        x &  ay & az & aw\\
                        w &  x  & ay & az\\
                        z  &  w  &  x   & ay   \\
                        y  &  z  &  w   & x   \\
                         \end{array}  
                        \right]}\in PGL_4(\overline{k}) \middle | x,y,z,w\in \overline{k} \right\}$. Note that $H$ is the centralizer of a $k$-anisotropic unipotent element $\overline{\left[\begin{array}{cccc}
                        0 &  0 & 0 & a\\
                        1 &  0  & 0 & 0\\
                        0  &  1  &  0   & 0   \\
                        0  &  0  &  1   & 0   \\
                         \end{array}  
                        \right]}$ of $PGL_4(k)$. Then $H$ is connected and $G$-ir over $k$. Note that $H$ is $3$-dimensional. Since $h^4=1$ for any $h\in H$, $H$ is a unipotent group. So, for an appropriate element $g\in G(\overline{k})$, $gHg^{-1}$ is a (possibly non-$k$-defined) subgroup of the $6$-dimensional group $U$ of upper unitriangular matrices of $G$. A computation by Magma (using the standard function \texttt{UpperTriangularMatrix($\cdot$)}) shows that there exist some $g\in G(\overline{k})$ such that\\
$g H g^{-1}= \left\{\overline{\left[\begin{array}{cccc}
                        X &  Y & Z & W\\
                        0 &  X  & Y & Z\\
                        0  &  0  &  X   & Y   \\
                        0  &  0  &  0   & X   \\
                         \end{array}  
                        \right]}\in PGL_4(\overline k) \middle | X,Y,Z,W\in \overline k\right\}$. 
Now for each $b\in \overline k$, define\\
  $H_b:=  \left \langle g H g^{-1},  
                        \overline{\left[\begin{array}{cccc}
                        1 &  0 & 0 & 0\\
                        0 &  1  & 0 & b\\
                        0  &  0  &  1   & 0   \\
                        0  &  0  &  0   & 1   \\
                         \end{array}  
                        \right] }\right \rangle$. Then a quick computation shows that the groups $H_b$ are distinct. So the groups $g^{-1} H_b g$ are infinitely many overgroups of $H$. 
\end{proof}

\begin{oprob}\label{openovergroup}
Let $k$ be a field. Let $G$ be semisimple algebraic group. Suppose that a $k$-subgroup $H$ of $G$ is connected and $G$-ir over $k$. Then does $H$ have only finitely many $k$-defined overgroups?
\end{oprob}

\section{The number of conjugacy classes of $k$-anisotropic unipotent elements}
Let $k=\overline{k}$. Let $G/k$ be conncected reductive. The following are known.
\begin{enumerate}
\itemsep0em 
\item {There are only finitely many conjugacy classes of unipotent elements in $G$~\cite[Thm.~13]{Lusztig-unipotent-invent}. }
\item {There is only a finite number $c_N$ of $G$-conjugacy classes of $G$-cr subgroups of fixed order $N$~\cite[Cor.~3.8]{Bate-geometric-Inventione}.}
\item {Moreover, if $G$ is simple, there is a uniform bound on $c_N$ that depends only on $N$ and the type of $G$ but not on $k$~\cite[Prop.~2.1]{Liebeck-Martin-Shalev}}.
\end{enumerate}  

Now let $k$ be nonperfect, and let $G/k$ be connected reductive. In this section, we show that the natural analogue of the above results 1, 2, 3 fail at the same time over a nonperfect $k$.  

\begin{prop}
Let $\mathbb{F}_2$ be the finite field with $2$ elements. Let $k:=\mathbb{F}_2(x)$ be a function field over $\mathbb{F}_2$. Let $G=PGL_2$. Then there exist infinitely many $G(k)$-conjugacy classes of $k$-anisotropic unipotent elements in $G$.
\end{prop}
\begin{proof}
Let $p_n(x)=x^{2\cdot3^n}+x^{3^n}+1\in k$ for $n\in\mathbb{N}$. Then each $p_n(x)$ is irreducible over $\mathbb{F}_2$ by~\cite[Ex.~3.96]{Lidl-Niederreiter-finitefields-book}. 
Let $u_n= \overline{\left[\begin{array}{cc}
                        0 & p_n(x) \\
                        1 & 0 \\
                         \end{array}  
                        \right]}$. 
It is clear that $u_n$ is a unipotent element of order $2$. Let $U_n:=\langle u_n \rangle$. Let $U_n$ act on $\mathbb{P}^1_k$ in the usual way. Since no eigenvalue of $u_n$ belongs to $k$, there is no non-trivial $k$-defined $U_n$-invariant proper subspace of $\mathbb{P}^1_k$. Thus $u_n$ is $k$-anisotropic.  

Suppose that $u_i$ is $PGL_2(k)$-conjugate to $u_j$ for some $j\neq i$. Then there exist $m\in GL_2(k)$ and $d\in k$ such that $m u_j m^{-1} =  \left[\begin{array}{cc}
                        d & 0 \\
                        0 & d \\
                         \end{array}  
                        \right] u_i$. Taking determinants on both sides, we obtain $p_j(x)=d^2 p_i(x)$. Let $d:=d_1/d_2$ where $d_1, d_2\in \mathbb{F}_2[x]$ and $d_1, d_2$ have no nontrivial common factor in $\mathbb{F}_2[x]$. Then $p_j(x)d_2^2=p_i(x)d_1^2$. So $d_1^2$ divides $p_j(x)$, but this is a impossible unless $d_1=1$ since $p_j(x)$ is irreducible. If $d_1=1$, we have $p_j(x)d_2^2=p_i(x)$. Then $d_2=1$ by the same argument. This is a contradiction since $p_i(x)\neq p_j(x)$.  
\end{proof}

\begin{rem}
The following was pointed out by the referee: it happens very frequently for  a non-algebraically closed $k$ that a simple group $G$ has infinitely many $G(k)$-conjugacy classes of unipotent elements (as an easy argument using the Galois cohomology shows). It is known that if a reductive group $G$ is defined over an local field $k$ and if the characteristic of $k$ is good for $G$, there are finitely many $G(k)$-conjugacy classes of unipotent elements. 
\end{rem}

\section*{Acknowledgements}
This research was supported by Marsden Grant UOA1021. The author would like to thank Benjamin Martin, Don Taylor, Philippe Gille, and Brian Conrad for helpful discussions. He is also grateful for detailed comments from an anonymous referee. 
\bibliography{mybib}

\end{document}